\documentclass[12pt]{amsart}
\usepackage{amsmath,amssymb}
\usepackage{amsfonts}
\usepackage{amsthm}
\usepackage{latexsym}
\usepackage{graphicx}

\def\R{\mathbb{R}}

\def\vv<#1>{\langle#1\rangle}

\def\1{\mathbf{1}}
\def\bA{\mathbf{A}}

\def\XXint#1#2{\setbox0=\hbox{$#1{#2}{\int}$}{#2}\kern-.5\wd0 }

\def\XXint#1#2#3{{\setbox0=\hbox{$#1{#2#3}{\int}$}
     \vcenter{\hbox{$#2#3$}}\kern-.5\wd0}}

\def\vv<#1>{{\left\langle#1\right\rangle}}

\def\sgn{{\rm sgn}}
\newtheorem{thm}{Theorem}[section]

\newtheorem{lem}{Lemma}[section]

\theoremstyle{definition}

\theoremstyle{remark}

\numberwithin{equation}{section}

\begin{document}
\title{Uniqueness of exterior differentiation on locally finite graphs}

\author{Yongjie Shi$^1$}
\address{Department of Mathematics, Shantou University, Shantou, Guangdong, 515063, China}
\email{yjshi@stu.edu.cn}
\author{Chengjie Yu$^2$}
\address{Department of Mathematics, Shantou University, Shantou, Guangdong, 515063, China}
\email{cjyu@stu.edu.cn}
\thanks{$^1$Research partially supported by NSF of China with contract no. 11701355. }
\thanks{$^2$Research partially supported by NSF of China with contract no. 11571215.}
\renewcommand{\subjclassname}{%
  \textup{2010} Mathematics Subject Classification}
\date{}
\begin{abstract}
In this short note, we prove the uniqueness of exterior differentiation on locally finite graphs.
\end{abstract}
\maketitle\markboth{Shi \& Yu}{Uniqueness of exterior differentiation}
\section{Introduction}
Let $M^n$ be a smooth manifold and $A^r(M)$ be the space of smooth exterior differential $r$-forms on $M$. Let $A^*(M)=\bigoplus_{r=0}^nA^r(M)$. It is a classical conclusion (see \cite{CCL}) that there is a unique operator $d:A^*(M)\to A^*(M)$ such that
\begin{enumerate}
\item $d$ is linear;
\item $d(A^r(M))\subset A^{r+1}(M)$;
\item $d(\alpha\wedge\beta)=d\alpha\wedge\beta+(-1)^{\deg\alpha}\alpha\wedge d\beta$ for any $\alpha,\beta\in A^*(M)$;
\item $d^2=0$;
\item $df(X)=X(f)$ for any smooth function $f$ and smooth vector field $X$.
\end{enumerate}

In \cite{SY}, by imitating theories in algebraic topology (\cite{Ha}) and Hodge theory for simplicial complexes (\cite{DHLM,Do}), we introduced exterior differential calculus on graphs. By comparing the theory with that of manifolds, there is a natural question about the uniqueness of exterior differentiation on graphs. In this short note, we answer this question by confirming the uniqueness of exterior differentiation on locally finite graphs. In this paper, for a graph, we mean a simple graph. A graph is said to be locally finite if the degree of each vertex is finite. We are still not sure if the conclusion remains true for graphs that are not locally finite in general.

Before stating the main result, let's recall some preliminaries in exterior differential calculus on graphs.

Let $G$ be a graph, and $V(G)$ and $E(G)$ be its sets of vertices and edges respectively. A $k$-clique of $G$ is a subset $\{v_1,v_2,\cdots, v_k\}$  of $V(G)$ with $k$ distinct elements such that $\{v_i,v_j\}\in E(G)$ for any $1\leq i<j\leq k$. The collection of all $k$-cliques of $G$ is denoted as $C_k(G)$. A $k$-tensor of $G$ is a function
$$f:\underbrace{V(G)\times V(G)\times\cdots\times V(G)}_{k+1}\to \R$$
such that  $f(v_0,v_1,\cdots,v_k)=0$ whenever $\{v_0,v_1,\cdots,v_k\}\not\in C_{k+1}(G)$. A $k$-form of $G$ is a $k$-tensor that is skew symmetric with respect to all of its variables. The collection of all $k$-forms on  $G$ is denoted as $A^k(G)$. Moreover, let $A^*(G)=\bigoplus_{k=0}^\infty A^k(G)$ which is the space of all differential forms on $G$. The tensor product $\alpha\otimes\beta$ of a $r$-tensor $\alpha$ and a $s$-tensor $\beta$ is a $(r+s)$-tensor defined as
\begin{equation}
\alpha\otimes\beta(v_0,v_1,\cdots,v_{r+s})=\alpha(v_0,v_1,\cdots,v_r)\beta(v_r,v_{r+1},\cdots,v_{r+s})
\end{equation}
for any $\{v_0,v_1,\cdots, v_{r+s}\}\in C_{r+s+1}(G)$. The skew symmetrization of a $r$-tensor $\alpha$ is defined as
\begin{equation}
\bA(\alpha)(v_0,v_1,\cdots,v_r)=\frac{1}{(r+1)!}\sum_{\sigma\in S_{r+1}}\sgn(\sigma)\cdot \alpha(v_{\sigma(0)},v_{\sigma(1)},\cdots,v_{\sigma(r)})
\end{equation}
where $S_{r+1}$ is the group of permutations on $\{0,1,2,\cdots,r\}$. Finally, the wedge product $\alpha\wedge\beta$ of a $r$-form $\alpha$ and a $s$-form $\beta$ is a $(r+s)$-form defined as
\begin{equation}
\alpha\wedge\beta=\bA(\alpha\otimes\beta).
\end{equation}
It was shown in \cite{SY} that the wedge product defined above for graphs has similar properties as those for differential manifolds:
\begin{enumerate}
\item For any $\alpha\in A^r(G)$ and $\beta\in A^s(G)$, $$\alpha\wedge\beta=(-1)^{rs}\beta\wedge\alpha.$$
\item For any $\alpha,\beta,\gamma\in A^*(G)$ with $d\alpha=d\beta=d\gamma=0$, $$(\alpha\wedge\beta)\wedge\gamma=\alpha\wedge(\beta\wedge\gamma).$$
\end{enumerate}
The difference here is that associativity for wedge product on graphs is only valid for closed forms. This is the same as that in algebraic topology and Hodge theory for simplicial complexes.

For a $r$-form $\alpha$, similar with the definition of boundary operator in algebraic topology, the exterior differential $d\alpha$ of $\alpha$ is a $r+1$-form defined as

\begin{equation}
\begin{split}
d\alpha(v_0,v_1,\cdots,v_{r+1})
=\sum_{i=0}^{r+1}(-1)^i\alpha(v_0,v_1,\cdots,\hat v_i,\cdots, v_{r+1}).
\end{split}
\end{equation}
for any $\{v_0,v_1,\cdots,v_{r+1}\}\in C_{r+2}(G)$. In \cite{SY}, we have shown that similar properties for $d$ with those for differential manifolds remain true for graphs:
\begin{enumerate}
\item $d^2=0$.
\item For any $\alpha\in A^r(G)$ and $\beta\in A^s(G)$,
$$d(\alpha\wedge\beta)=d\alpha\wedge\beta+(-1)^r\alpha\wedge d\beta.$$
\end{enumerate}
The main result of this paper says that the exterior differential operator given above is the only operator on graphs satisfying the properties above which extends the similar result for differential manifolds to the discrete case.
\begin{thm}\label{thm-uniqueness}
Let $G$ be a locally finite graph. Let $\delta:A^*(G)\to A^*(G)$ be a linear operator such that
\begin{enumerate}
\item $\delta(A^r(G))\subset A^{r+1}(G)$.
\item $\delta^2=0$.
\item For any $\alpha\in A^r(G)$ and $\beta\in A^s(G)$,
$$\delta(\alpha\wedge\beta)=\delta\alpha\wedge\beta+(-1)^r\alpha\wedge\delta\beta.$$
\item $\delta f=df$ for any $f\in A^0(G)$ .
\end{enumerate}
Then, $\delta=d$.
\end{thm}
We will imitate the proof for differential manifolds. Let's recall the proof of the corresponding result  for differential manifolds. Let $\delta:A^*(M)\to A^*(M)$ be an operator satisfying the properties.  For any differential $r$-form $\alpha$ on a differential manifold, one can write $\alpha$ as
\begin{equation}
\begin{split}
\alpha=&\sum_{1\leq i_1<i_2<\cdots<i_r\leq n}\alpha_{i_1i_2\cdots i_r}(x)dx^{i_1}\wedge dx^{i_2}\wedge\cdots\wedge dx^{i_r}\\
=&\sum_{1\leq i_1<i_2<\cdots<i_r\leq n}\alpha_{i_1i_2\cdots i_r}(x)\delta x^{i_1}\wedge \delta x^{i_2}\wedge\cdots\wedge \delta x^{i_r}
\end{split}
\end{equation}
locally under a local coordinate $(x^1,x^2,\cdots,x^n)$, because $df=\delta f$ for any smooth function $f$ by assumption. By the linearity and that $\delta^2=0$ and the Leibnitz rule, one has
\begin{equation}
\begin{split}
\delta\alpha=&\sum_{1\leq i_1<i_2<\cdots<i_r\leq n}\delta\alpha_{i_1i_2\cdots i_r}\wedge dx^{i_1}\wedge dx^{i_2}\wedge\cdots\wedge dx^{i_r}\\
=&\sum_{1\leq i_1<i_2<\cdots<i_r\leq n}d\alpha_{i_1i_2\cdots i_r}\wedge dx^{i_1}\wedge dx^{i_2}\wedge\cdots\wedge dx^{i_r}\\
=&d\alpha
\end{split}
\end{equation}
because $df=\delta f$ for any smooth function $f$.

For the functions on a graph playing the role of local coordinate functions, we choose the characteristic function at each vertex of $G$ : for each $v\in V(G)$, let
\begin{equation}
\chi^v(x)=\left\{\begin{array}{ll}1&x=v\\
0&x\neq v.
\end{array}\right.
\end{equation}
The key ingredient in proving Theorem \ref{thm-uniqueness} is the following lemma essentially says that $A^0(G)$ and $\{d\chi^v\ |\ v\in V(G)\}$ generate $A^*(G)$.
\begin{lem}\label{lem-expansion}
Let $G$ be a graph and $\alpha$ be $k$-form on G with $k\geq 1$. Then,
\begin{equation}\label{eq-expansion}
\alpha=\sum_{v_1,v_2,\cdots,v_k\in V(G)}\alpha_{v_1v_2\cdots v_k}\wedge\left(d\chi^{v_1}\wedge\cdots\wedge d\chi^{v_k}\right)
\end{equation}
where $\alpha_{v_1v_2\cdots v_k}(x)=\alpha(x,v_1,v_2,\cdots,v_k)$ for any $x,v_1,\cdots,v_k\in V(G)$.
\end{lem}
Although the right hand side of \eqref{eq-expansion} is an infinite sum  when the graph is infinite, for any $x_0,x_1,\cdots, x_k\in V(G)$, the set
$$\left\{(v_1,v_2,\cdots,v_k)\ |\ \left(\alpha_{v_1v_2\cdots v_k}\wedge\left(d\chi^{v_1}\wedge\cdots\wedge d\chi^{v_k}\right)\right)(x_0,x_1,\cdots,x_k)\neq 0\right\}$$
is finite. So, only finite sum is involved in the right hand side of \eqref{eq-expansion} when applying it to $k+1$ given vertices.
\section{Proof of the main result}
First, we have following two simple lemmas.
\begin{lem}\label{lem-wedge-product}
Let $G$ be a graph, $f\in A^0(G)$ and $\alpha\in A^k(G)$. Then, for any $v_0,v_1,\cdots, v_k\in V(G)$,
\begin{equation}
(f\wedge\alpha)(v_0,v_1,\cdots,v_k)=\frac{1}{k+1}\left(\sum_{i=0}^{k}f(v_i)\right)\alpha(v_0,v_1,\cdots,v_k).
\end{equation}
\end{lem}
\begin{proof}
By the definition of wedge product on graphs,
\begin{equation}
\begin{split}
&(f\wedge\alpha)(v_0,v_1,\cdots,v_k)\\
=&\bA(f\otimes\alpha)(v_0,v_1,\cdots,v_k)\\
=&\frac{1}{(k+1)!}\sum_{\sigma\in S_{k+1}}\sgn\sigma f(v_{\sigma(0)})\alpha(v_{\sigma(0)},v_{\sigma(1)},\cdots,v_{\sigma(k)})\\
=&\frac{1}{(k+1)!}\sum_{\sigma\in S_{k+1}} f(v_{\sigma(0)})\alpha(v_{0},v_{1},\cdots,v_{k})\\
=&\frac{1}{k+1}\left(\sum_{i=0}^{k}f(v_i)\right)\alpha(v_0,v_1,\cdots,v_k).
\end{split}
\end{equation}
\end{proof}
\begin{lem}\label{lem-wedge-product-2}
Let $G$ be a graph. Then, for any $v_1,v_2,\cdots,v_k\in V(G)$ and $x_0,x_1,\cdots,x_k\in V(G)$ with $\{x_0,x_1,\cdots,x_{k}\}\in C_{k+1}(G)$,
\begin{equation}
\begin{split}
&d\chi^{v_1}\wedge d\chi^{v_2}\wedge\cdots\wedge d\chi^{v_k}(x_0,x_1,\cdots,x_k)\\
=&\left\{\begin{array}{ll}\frac{\sgn\tau}{k!}&\tau\in S_{k+1}\ {\rm such\ that}\ v_i=x_{\tau(i)}\ {\rm for}\ i=1,2,\cdots,k\\
0&{\rm otherwise.}
\end{array}\right.
\end{split}
\end{equation}
\end{lem}
\begin{proof}First, it is clear that if $v_1,v_2,\cdots,v_k$ are not distinct, then
$$d\chi^{v_1}\wedge d\chi^{v_2}\wedge\cdots\wedge d\chi^{v_k}=0.$$
Next, when $v_1,v_2,\cdots,v_k$ are distinct, by the definition of wedge product, we have
\begin{equation}\label{eq-wedge-product-1}
\begin{split}
&d\chi^{v_1}\wedge d\chi^{v_2}\wedge\cdots\wedge d\chi^{v_k}(x_0,x_1,\cdots,x_k)\\
=&\frac{1}{(k+1)!}\sum_{\sigma\in S_{k+1}}\sgn\sigma \prod_{i=1}^k(\chi^{v_i}(x_{\sigma(i)})-\chi^{v_i}(x_{\sigma(i-1)}))\\
=&\frac{1}{(k+1)!}\sum_{\sigma\in S_{k+1}}\sgn\sigma \sum_{i=0}^{k} (-1)^i\chi^{v_1v_2\cdots v_k}(x_{\sigma(0)},x_{\sigma(1)},\cdots,\widehat{ x_{\sigma(i)}},\cdots,x_{\sigma(k)})
\end{split}
\end{equation}
where $\chi^{v_1v_2\cdots v_k}(u_1,u_2,\cdots,u_k):=\chi^{v_1}(u_1)\cdots\chi^{v_k}(u_k)$. So, if there is some $v_i\not\in \{x_0,x_1,\cdots,x_k\}$, then
$$d\chi^{v_1}\wedge d\chi^{v_2}\wedge\cdots\wedge d\chi^{v_k}(x_0,x_1,\cdots,x_k)=0.$$
When $\{v_1,v_2,\cdots,v_k\}\subset \{x_0,x_1,\cdots,x_k\}$, suppose that
\begin{equation}
v_i=x_{\tau(i)}
\end{equation}
for $i=1,2,\cdots,k$ with $\tau\in S_{k+1}$. Then, by \eqref{eq-wedge-product-1},
\begin{equation}\label{eq-wedge-product-1}
\begin{split}
&d\chi^{v_1}\wedge d\chi^{v_2}\wedge\cdots\wedge d\chi^{v_k}(x_0,x_1,\cdots,x_k)\\
=&\frac{1}{(k+1)!}\sum_{\sigma\in S_{k+1}}\sgn\sigma \sum_{i=0}^{k} (-1)^i\chi^{x_{\tau(1)}x_{\tau(2)}\cdots x_{\tau(k)}}(x_{\sigma(\mu(1))},x_{\sigma(\mu(2))},\cdots,\cdots,x_{\sigma(\mu(k))})\\
=&\frac{1}{(k+1)!}\sum_{i=0}^{k}\sum_{\sigma\in S_{k+1}}\sgn(\sigma\mu) \chi^{x_{\tau(1)}x_{\tau(2)}\cdots x_{\tau(k)}}(x_{\sigma(\mu(1))},x_{\sigma(\mu(2))},\cdots,x_{\sigma(\mu(k))})\\
=&\frac{1}{k!}\sgn\tau.
\end{split}
\end{equation}
Here $\mu\in S_{k+1}$ is given by $\mu(0)=i,\mu(1)=0,\mu(2)=1,\cdots,\mu(i)=i-1,\mu(i+1)=i+1,\cdots,\mu(k)=k$, and it is clear that $\sgn \mu=(-1)^i$.
\end{proof}
We are now ready to prove Lemma \ref{lem-expansion}.
\begin{proof}By Lemma \ref{lem-wedge-product} and lemma \ref{lem-wedge-product-2}, for any $x_0,x_1,\cdots,x_k\in V(G)$ with $\{x_0,x_1,\cdots,x_k\}\in C_{k+1}(G)$,
\begin{equation}
\begin{split}
&\sum_{v_1,v_2,\cdots,v_k\in V(G)}\alpha_{v_1v_2\cdots v_k}\wedge\left(d\chi^{v_1}\wedge\cdots\wedge d\chi^{v_k}\right)(x_0,x_1,\cdots,x_k)\\
=&\frac{1}{k+1}\sum_{v_1,v_2,\cdots,v_k\in V(G)}\left(\sum_{i=0}^k\alpha(x_i,v_1,v_2,\cdots,v_k)\right)\left(d\chi^{v_1}\wedge\cdots\wedge d\chi^{v_k}\right)(x_0,x_1,\cdots,x_k)\\
=&\frac1{(k+1)!}\sum_{\tau\in S_{k+1}}\sum_{i=0}^k\sgn\tau\alpha(x_i,x_{\tau(1)},x_{\tau(2)},\cdots,x_{\tau(k)})\\
=&\alpha(x_0,x_1,\cdots,x_k).
\end{split}
\end{equation}
This completes the proof of the lemma.
\end{proof}
Because $$\sum_{v_1,v_2,\cdots,v_k\in V(G)}\alpha_{v_1v_2\cdots v_k}\wedge\left(d\chi^{v_1}\wedge\cdots\wedge d\chi^{v_k}\right)$$ is formally an infinite sum of $k$-forms when $G$ is an infinite graph, we can not interchange the order of $\delta$ and $\Sigma$ by using the linearity of $\delta$ in proving $\delta=d$. Similar to the case of differential manifolds, we first cut $\alpha$ off to ensure the summation is a finite sum.
\begin{lem}\label{lem-cut-off}
Let $G$ be a graph and $\delta$ be an operator satisfying the assumptions in Theorem \ref{thm-uniqueness}. For $\{x_0,x_1,\cdots,x_{k+1}\}\in C_{k+2}(G)$, let $\rho\in A^0(G)$ be such that $\rho(x_i)=1$ for $i=0,1,2,\cdots,k+1$ and $\rho(x)=0$ when $x\not\in \{x_0,x_1,\cdots,x_{k+1}\}$. Then,
\begin{equation}
\delta(\rho\wedge\alpha)(x_0,x_1,\cdots,x_{k+1})=\delta\alpha(x_0,x_1,\cdots,x_{k+1})
\end{equation}
for any $\alpha\in A^k(G)$.
\end{lem}
\begin{proof}
By the assumption for Leibnitz rule in Theorem \ref{thm-uniqueness} and Lemma \ref{lem-wedge-product}, we have
\begin{equation}
\begin{split}
&\delta(\rho\wedge\alpha)(x_0,x_1,\cdots,x_{k+1})\\
=&\delta\rho\wedge\alpha(x_0,x_1,\cdots,x_{k+1})+\rho\wedge\delta\alpha(x_0,x_1,\cdots,x_{k+1})\\
=&\delta\alpha(x_0,x_1,\cdots,x_{k+1}).
\end{split}
\end{equation}
Here $\delta\rho\wedge\alpha(x_0,x_1,\cdots,x_{k+1})=0$ because $\delta\rho(x_i,x_j)=\rho(x_j)-\rho(x_i)=0$ for any $i,j=0,1,\cdots,k+1$.
\end{proof}
We are now ready to prove Theorem \ref{thm-uniqueness}.
\begin{proof}[Proof of Theorem \ref{thm-uniqueness}] Let $\alpha\in A^k(G)$ and $\{x_0,x_1,\cdots,x_{k+1}\}\in C_{k+2}(G)$. Let $\rho$ be the function in Lemma \ref{lem-cut-off}. Then, by Lemma \ref{lem-cut-off} and Lemma \ref{lem-expansion}
\begin{equation}\label{eq-delta}
\begin{split}
&\delta\alpha(x_0,x_1,\cdots,x_{k+1})\\
=&\delta(\rho\wedge\alpha)(x_0,x_1,\cdots,x_{k+1})\\
=&\delta\left(\sum_{v_1,v_2,\cdots,v_k\in V(G)}(\rho\wedge\alpha)_{v_1v_2\cdots v_k}\wedge(d\chi^{v_1}\wedge d\chi^{v_2}\wedge\cdots\wedge d\chi^{v_k})\right)(x_0,x_1,\cdots,x_{k+1}).\\
\end{split}
\end{equation}
By Lemma \ref{lem-wedge-product},
\begin{equation}
\begin{split}
&(\rho\wedge\alpha)_{v_1v_2\cdots v_k}(x)\\
=&\rho\wedge\alpha(x,v_1,v_2,\cdots,v_k)\\
=&\frac{1}{k+1}(\rho(x)+\rho(v_1)+\cdots+\rho(v_k))\alpha(x,v_1,\cdots,v_k)\\
=&0
\end{split}
\end{equation}
if there is some $v_i$ with $d(v_i,x_j)\geq 2$ for any $j=0,1,2,\cdots,k+1$. Indeed, suppose that $d(v_1,x_j)\geq 2$ for any $j=0,1,2\cdots,k+1$. If $\alpha(x,v_1,\cdots,v_k)\neq 0$, then
$$d(x,v_1)=d(v_2,v_1)=\cdots=d(v_k,v_1)=1.$$
This implies that
$$\{x,v_1,v_2,\cdots,v_k\}\cap \{x_0,x_1,\cdots,x_{k+1}\}=\emptyset.$$
So,
\begin{equation}
\rho(x)+\rho(v_1)+\cdots+\rho(v_k)=0.
\end{equation}
Therefore
\begin{equation}
\begin{split}
&\{(v_1,v_2,\cdots,v_k)\ |\ (\rho\wedge\alpha)_{v_1v_2\cdots v_k}\not\equiv 0.\}\\
\subset& \{(v_1,v_2,\cdots,v_k)\ |\ d(v_i,x_j)\leq 1\ {\rm for\ some}\ j=0,1,2,\cdots,k+1.\}
\end{split}
\end{equation}
is a finite set because $G$ is locally finite. Then, by linearity of $\delta$,
\begin{equation}
\begin{split}
&\delta\alpha(x_0,x_1,\cdots,x_{k+1})\\
=&\left(\sum_{v_1,v_2,\cdots,v_k\in V(G)}\delta(\rho\wedge\alpha)_{v_1v_2\cdots v_k}\wedge(d\chi^{v_1}\wedge d\chi^{v_2}\wedge\cdots\wedge d\chi^{v_k})\right)(x_0,x_1,\cdots,x_{k+1})\\
=&\left(\sum_{v_1,v_2,\cdots,v_k\in V(G)}d(\rho\wedge\alpha)_{v_1v_2\cdots v_k}\wedge(d\chi^{v_1}\wedge d\chi^{v_2}\wedge\cdots\wedge d\chi^{v_k})\right)(x_0,x_1,\cdots,x_{k+1})\\
=&d\left(\sum_{v_1,v_2,\cdots,v_k\in V(G)}(\rho\wedge\alpha)_{v_1v_2\cdots v_k}\wedge(d\chi^{v_1}\wedge d\chi^{v_2}\wedge\cdots\wedge d\chi^{v_k})\right)(x_0,x_1,\cdots,x_{k+1})\\
=&d(\rho\wedge\alpha)(x_0,x_1,\cdots,x_{k+1})\\
=&d\alpha(x_0,x_1,\cdots,x_{k+1}).
\end{split}
\end{equation}
This completes the proof of the theorem.
\end{proof}

\end{document}